\documentclass[12pt]{amsart}

\usepackage{amssymb}

\usepackage{enumitem}

\usepackage{graphicx}

\usepackage{lipsum}

\usepackage{stmaryrd}
\usepackage{graphicx}
\usepackage{amscd}
\usepackage{amsmath}
\usepackage{amsfonts}
\usepackage{mathrsfs}
\usepackage{amssymb}
\usepackage{verbatim}
\usepackage{color}
\usepackage{hyperref}

\makeatletter
\@namedef{subjclassname@2020}{%
  \textup{2020} Mathematics Subject Classification}
\makeatother

\usepackage[T1]{fontenc}

\newtheorem{theorem}{Theorem}[section]
\newtheorem{corollary}[theorem]{Corollary}
\newtheorem{lemma}[theorem]{Lemma}
\newtheorem{proposition}[theorem]{Proposition}

\theoremstyle{definition}
\newtheorem{definition}[theorem]{Definition}
\newtheorem{remark}[theorem]{Remark}
\newtheorem{example}[theorem]{Example}

\numberwithin{equation}{section}

\def\what{\widehat}

\def\One{{1\!\!1}}

\def\wtil{\widetilde}
\def\half{\frac{1}{2}}

\def\Pref{{\rm Pref}}
\newcommand{\lam}{\lambda}

\newcommand{\om}{\omega}

\newcommand{\sig}{\sigma}

\def\bbe{{\bf e}}
\newcommand{\R}{{\mathbb R}}
\newcommand{\Q}{{\mathbb Q}}
\newcommand{\Z}{{\mathbb Z}}
\newcommand{\C}{{\mathbb C}}

\def\N{{\mathbb N}}

\newcommand{\Nat}{{\mathbb N}}

\def\wt{\widetilde}
\def\A{{\mathcal A}}

\def\Rk{{\mathcal R}}
\def\Qk{{\mathcal Q}}

\def\Sf{{\sf S}}

\def\T{{\mathbb T}}

\def\bt{{\mathbf t}}
\def\bz{{\mathbf z}}

\def\be{\begin{equation}}
\def\ee{\end{equation}}

\newcommand{\eps}{{\varepsilon}}
\newcommand{\es}{\emptyset}

\newcommand{\const}{{\rm const}}

\def\Cc{{\mathscr M}}
\def\Mc{{\mathscr M}}

\def\a{a}

\def\ve1{\vec{1}}

\def\Ak{{\mathcal A}}

\def\bn{{\bf n}}
\def\mfr{\mathfrak m}

\begin{document}


\baselineskip=17pt


\title[Lyapunov spectrum of the twisted cocycle for substitutions]{On the Lyapunov spectrum of the twisted  \\[1.1ex]cocycle for substitutions}

\author{Boris Solomyak}

\address{Boris Solomyak\\ Department of Mathematics,
Bar-Ilan University, Ramat-Gan, Israel}
\email{bsolom3@gmail.com}
\date{}

\begin{abstract}
The paper is devoted to the properties of a complex matrix ``twisted,'' otherwise called ``spectral,'' cocycle, 
associated with substitution dynamical systems. Following a recent finding 
of Rajabzadeh and Safaee \cite{RS2} of an invariant section for the twisted cocycle, we indicate that this implies presence of a zero Lyapunov
exponent. This has consequences for the spectral properties of substitution dynamical systems; in particular, this extends the scope and
simplifies the proof of singular spectrum for a large class of substitutions on two symbols. We also obtain some results on positivity of the top
exponent. In the appendix we compute the Lebesgue almost everywhere local dimension of spectral measures of some ``simple'' test functions, for
almost every irrational rotation. This sheds some light on the earlier work of Bufetov and the author \cite{BuSo20}, relating the
local dimension of spectral measures to pointwise Lyapunov exponents of the twisted cocycle.
It should be noted that the paper has some (mutually acknowledged) overlap with \cite[Appendix]{RS2}.

\end{abstract}



\maketitle

\bigskip

\thispagestyle{empty}

\section{Introduction}

We continue the investigation of a complex matrix cocycle associated with a class of dynamical systems of parabolic type, such as
substitutions, $S$-adic systems, interval exchange transformations (IET's), and translation flows.
In particular, it proved to be useful in the study of quantitative weak mixing and spectral properties, such  as singular spectrum.
It appeared implicitly, as a {\em generalized matrix Riesz product} in \cite{BuSo14,BuSo18} and defined explicitly in \cite{BuSo20}, under the name
{\em spectral cocycle}. Some of the more recent studies in which the spectral cocycle played an important role
include \cite{BuSo22,Marshall24a,Marshall24b,Sol24}. 
A closely related, {\em twisted cocycle} 
over the Teichm\"uller translation flow, was introduced by 
Forni \cite{Forni22}. Avila, Forni, and Safaee \cite{AFS24} employed the twisted cocycle, essentially its $S$-adic version, in their work on
quantitative weak mixing for typical IET's. 
In a parallel development, a closely related object, named {\em Fourier cocycle}, was used in the study of diffraction spectrum for substitution tilings in the work of Baake et al. \cite{BFGR19,BGM19}. Here we focus on the class of substitution systems and adopt the term {\em ``twisted cocycle''}.
We do not attempt to provide an exhaustive bibliography, referring the reader to the surveys \cite{BuSo23,Forni24} 
and the recent preprint \cite{RS2}.

The goal of this note is to point out some consequences of a recent discovery by 
Rajabzadeh and Safaee \cite{RS2} of an invariant section for the twisted cocycle,
which implies  the presence of a zero Lyapunov exponent in the Lyapunov spectrum. Whereas the focus of \cite{RS2} is on interval exchanges, it was
realized that this phenomenon is more general; in fact, it applies to arbitrary $S$-adic systems and hence to substitutions. We should point out  that there
remain many open questions on Lyapunov spectrum of the twisted cocycle. It was shown in \cite{RS1} that for IET's the top Lyapunov exponent
of the twisted cocycle, defined over the toral extension of the Zorich (Rauzy-Veech) renormalization, is positive. In contrast,
 for a single substitution, where
the twisted cocycle is over the toral endomorphism induced by the transpose substitution matrix, little is known.

The paper is organized as follows. After introducing the background, we state the result on the invariant section and zero Lyapunov exponent, with
a quick proof. This is followed by applications of two kinds: 1) showing singular spectrum for a wide class of substitutions on 2 symbols; 2) providing
some necessary, and separately sufficient conditions for the Lyapunov spectrum to be degenerate. For constant length substitutions on 2 symbols
this has been done earlier in \cite{BCM20,Marshall24a}.
  

\section{Preliminaries}

\subsection{Substitution dynamical systems}
The standard references for the basic facts on substitution dynamical systems are \cite{Queff,Siegel}.
Consider an alphabet of $d\ge 2$ symbols $\Ak=\{0,\ldots,d-1\}$. Let  $\A^+$ be the set of nonempty words with letters in $\A$. 
A {\em substitution}  is a map $\zeta:\,\A\to \A^+$, extended to 
 $\A^+$ and $\A^{\N}$ by
concatenation. The {\em substitution space} is defined as the set of bi-infinite sequences $x\in \A^\Z$ such that any word  in $x$
appears as a subword of $\zeta^n(\a)$ for some $\a\in \A$ and $n\in \N$. The {\em substitution dynamical system}  is the left
shift on $\A^\Z$ restricted to $X_\zeta$, which we denote by $T$.
The {\em substitution matrix} $\Sf=\Sf_\zeta=(\Sf(i,j))$ is the $d\times d$ matrix, such that $\Sf(i,j)$ is the number
of symbols $i$ in $\zeta(j)$. The substitution is {\em primitive} if $\Sf_\zeta^n$ has all entries strictly positive for some $n\in \Nat$.
It is well-known that primitive substitution $\Z$-actions are minimal and uniquely ergodic, see \cite{Queff}.
We assume that the substitution is primitive and {\em non-periodic}, which in the primitive case is equivalent to the space $X_\zeta$ being infinite.
The length of a word $u$ is denoted by $|u|$. The substitution $\zeta$ is said to be of {\em constant length} $q$ if $|\zeta(a)|=q$ for all $a\in \A$, otherwise, it is of {\em non-constant length}.

\subsection{Twisted (spectral) cocycle}

Write $\bz = (z_0,\ldots,z_{d-1})$ and $\bz^v = z_{v_0}z_{v_1}\ldots z_{v_k}$ for a word $v = v_0v_1\ldots v_k\in \A^{k+1}$.
For a word $u\in \A^+$ and $a\in \A$ consider the polynomial in $\bz$-variables which keeps track of the returns to $a$:
\be \label{poly-return}
P_{u,a}(\bz) = \sum_{j\le |u|,\ u_j = a} \bz^{u_1\ldots u_{j-1}},
\ee
where $j=1$ corresponds to $\bz^{\es}=1$. 

\begin{example}
Let $u=01001011$. Then  $$P_{u,0}(\bz) = 1+z_0z_1 + z_0^2z_1+ z_0^3z_1^2, \ \ \ \ P_{u,1}(\bz) = z_0 + z_0^3z_1 + z_0^4z_1^2 + z_0^4 z_1^3.$$
\end{example}

Given a substitution $\zeta:\Ak\to \Ak^+$, suppose that
$\zeta(b) = u_1^{(b)}\ldots u_{k_b}^{(b)}$ for $b\in \Ak.$
Define a matrix-function on the $d$-torus $(S^1)^d$ as follows:
\be \label{def-matrix0}
\Mc_\zeta(\bz) = [\Mc_\zeta(z_0,\ldots,z_{d-1})]_{b,c} = {\Bigl( P_{\zeta(b),c} (\bz)\Bigr)}_{(b,c)\in \A^2},\ \ \ \bz\in (S^1)^d.
\ee

We are going to lift $\Mc_\zeta$ to the universal cover, in other words, write $z_j = \exp(-2\pi i \xi_j)$ and 
 $\xi = (\xi_0,\ldots,\xi_{d-1})$. Thus we obtain a $\Z^d$-periodic matrix-valued
 function, which we denote by the same letter.
 It can be written explicitly as follows:
  $\Cc_\zeta:  \R^d\to M_d(\C)$  (the space of complex $d\times d$ matrices): 
 \be \label{coc0}
\Cc_\zeta(\xi) = [\Cc_\zeta(\xi_0\ldots,\xi_{d-1})]_{(b,c)} := \Bigl(\!\!\!\! \!\!\!\!\!\sum_{j\le |\zeta(b)|,\ u_j^{(b)} = c}\!\!\!\!\!\!\! \exp\bigl(-2\pi i \sum_{k=1}^{j-1} \xi_{u_k^{(b)}}\bigr)\Bigr)_{(b,c)\in \A^2},\ \xi\in \R^d.
\ee

\begin{example}
Consider the substitution $\zeta: 0\mapsto 01200,\ 1\mapsto 120,\ 2 \mapsto 110$. Then
$$
\Mc_{\zeta}(z_0,z_1,z_2) = \left( \begin{array}{ccc} 1 + z_0z_1z_2 + z_0^2z_1z_2 & z_0 & z_0z_1 \\ z_1z_2 & 1 & z_1 \\ z_1^2 & 1+z_1 & 0 \end{array} 
\right), \ \ \ z_j = e^{-2\pi i \xi_j}.
$$
\end{example}

\medskip

Note that
$\Mc_\zeta(0) = \Sf_\zeta^{\sf T}$.
The entries of the matrix $\Mc_\zeta(\xi)$ are trigonometric polynomials with coefficients 0's and 1's that  are less than or equal to 
 the corresponding entries of $\Sf_\zeta^{\sf T}$
 in absolute value, for every $\xi\in \T^d$.
Crucially, the 
{\em cocycle condition} holds:
 for any two substitutions $\zeta_1,\zeta_2$ on the same alphabet,
\be \label{most}
\Cc_{\zeta_1\circ \zeta_2}(\xi) = \Cc_{\zeta_2}(\Sf^{\sf T}_{\zeta_1}\xi)\Cc_{\zeta_1}(\xi),
\ee
which is verified by a direct computation.

\begin{definition}
Suppose that $\det(\Sf_\zeta)\ne 0$, and consider the  endomorphism of the torus $\T^d$
\begin{equation}\label{torend}
E_\zeta: \xi \mapsto \Sf_\zeta^{\sf T} \xi \  (\mathrm{mod} \   \Z^d),
\end{equation} 
 which preserves the Haar measure $m_d$. Then
 \be \label{cocycle3}
\Cc_\zeta(\xi,n):= \Cc_\zeta\bigl((\Sf_\zeta^{\sf T})^{n-1}\xi \bigr)\cdots\Cc_\zeta(\xi),
\ee
  is called the {\em twisted (spectral) cocycle}, associated to $\zeta$, over the endomorphism \eqref{torend}.
 \end{definition}
 
Note that (\ref{most}) implies
\be \label{coc2}
\Mc_\zeta(\xi,n) =\Cc_{\zeta^n}(\xi),\ \ n\in \N.
\ee

Consider the pointwise upper  Lyapunov exponent of the cocycle at $\xi\in \T^d$, defined by
\be \label{Lyap2}
{\chi}_{\zeta,\xi}^+= \limsup_{n\to \infty} \frac{1}{n} \log \|\Cc_\zeta(\xi,n)\|.
\ee
If the limit exists, we drop the superscript ``+''.
Since  for every $\xi$ the absolute values of the entries of $\Cc_\zeta(\xi,n)$ are not greater than those of $\Sf_{\zeta^n}^{\sf T}$, we have
$$\chi_{\zeta,\xi}^+ \le \log\theta\ \ \mbox{ for all}\ \xi\in \T^d,$$
where $\theta$ is the Perron-Frobenius (PF) eigenvalue of $\Sf_\zeta$. The pointwise upper Lyapunov exponent $\chi_{\zeta,\xi}^+$ is invariant under the action of the endomorphism $E_\zeta$.

Now suppose that the endomorphism is ergodic,
that is, $\Sf_\zeta$ has no eigenvalues that are roots of unity (see \cite[Cor,\,2.20]{EWergbook}). 
Then by theorems of Furstenberg-Kesten \cite{FK} and Kingman \cite{Kingman}, there is a ``global'' Lyapunov exponent
\begin{eqnarray} \nonumber
\chi_{_{\scriptstyle{\zeta}}} = \chi_{_{\scriptstyle \zeta,\max}} & = & \lim_{n\to \infty} \frac{1}{n} \log \|\Cc_\zeta(\xi,n)\|,\ \ \mbox{for $m_d$-a.e.}
\ \xi\in \T^d\\[1.2ex]
& = &  \inf_n \frac{1}{n} \int_{\T^d} \log\|\Cc_\zeta(\xi,n)\|\,dm_d(\xi). \label{Lyap12}
\end{eqnarray}

The value of the Lyapunov exponent does not depend on the norm; it will be convenient to use the Frobenius norm of a matrix, defined by
$${\|(a_{ij})_{i,j}\|}^2_{\rm F} = \sum_{i,j} |a_{ij}|^2.$$ 

\begin{lemma}[{see e.g., \cite[Lemma 2.3]{BuSo22}}] \label{lem:lower}
For any $n\ge 1$, the function $\xi \mapsto \log\|\Cc_\zeta(\xi),n)\|$ is integrable, and
$$
\int_{\T^d} \log\|\Cc_\zeta(\xi,n)\|\,dm_d(\xi)\ge 0.
$$
Thus, $\chi_{_{\scriptstyle{\zeta}}}\ge 0$.
\end{lemma}

\begin{proof} We recall the argument for the reader's convenience. By Definition~\ref{coc0},  we obtain that for any substitution $\zeta$ on $\Ak$, the function ${\|\Cc_\zeta(\xi)\|}_{\rm F}^2$ is a sum of the squares of absolute values of
polynomials in $d$ variables $z_j = e^{-2\pi i \xi_j}$. This expression can be rewritten as a polynomial in the variables $z_j^{\pm 1}$. Multiplying it by
 $z_j^{\ell_j}$ for some $\ell_j\ge 0$, so as to get rid of the negative powers, we obtain that 
$$
 \log {\|\Cc_\zeta(\xi)\|}_{\rm F}^2 = \log|P_\zeta(z_0,\ldots,z_{d-1})|,
$$
 for some polynomial $P_\zeta$ with integer coefficients. Therefore,
 $$
 \int_{\T^d} \log {\|\Cc_\zeta(\xi)\|}_{\rm F}^2\,dm_d(\xi) = \mfr(P_\zeta)\ge 0,
$$
 where $\mfr(P_\zeta)$ is the logarithmic Mahler measure of a polynomial, see the definition below. Since
 $
 \Cc_\zeta(\xi,n) = \Cc_{\zeta^{n}}(\xi),
 $
the claim follows.
\end{proof}

By the Oseledets Theorem (one-sided, unless $\det(\Sf_\zeta) = \pm 1$,
in which case $E_\zeta$ is a toral automorphism), the entire Lyapunov spectrum is well-defined.
We will denote the maximal and the minimal Lyapunov exponents of the twisted cocycle
by $\chi_{_{\scriptstyle{\zeta}}} = \chi_{_{\scriptstyle \zeta,\max}}$ and
$\chi_{_{\scriptstyle \zeta,\min}}$ respectively.

\subsection{Mahler measure of a polynomial}

Mahler \cite{Mahler} defined the  measure of a polynomial in $d$ variables as follows:
$$
M(P):= \exp \int_{\T^d} \log|P(z_0,\ldots,z_{d-1})| \,d\bt,
$$
where $\bt = (t_0,\ldots,t_{d-1})$ and $z_j = e^{2\pi i t_j}$. The number $M(P)$ is called the {\em Mahler measure} and  
\be \label{eq-Mahler}
\mfr(P) = \log M(P)
\ee
is the {\em logarithmic Mahler measure} of the polynomial $P$. It is known that $M(P)$ is well-defined (i.e., $\bz \mapsto \log|P(\bz)|$ is integrable on 
$\T^d$), and $M(P)\ge 1$, provided that $P$ has integer coefficients, see \cite[Lemma 3.7]{EW}.
Moreover, there is a characterization when $M(P)=1$ for an integer polynomial $P$.
 In the single-variable case it is well-known from Kronecker's Lemma that $M(P)=1$ if and only if all roots of $P$ are of modulus one, i.e., $P$ is a 
cyclotomic polynomial. We need a generalization to the multi-variable case, see \cite[Theorem 3.10]{EW}. Several equivalent formulations (and different
proofs) are available; the next definition is from Smyth \cite{Smy81}.

\begin{sloppypar}
\begin{definition} \label{def:cyc}
An extended cyclotomic polynomial is a polynomial of the form 
\be \label{eq:cyc}
\psi(\bz) = z_0^{b_0}\cdots z_{d-1}^{b_{d-1}} \Phi(z_0^{v_0}\ldots z_{d-1}^{v_{d-1}}),
\ee 
where
$\Phi$ is a cyclotomic polynomial, $v_i$ are integers, not all of them equal to zero,
and $b_i= \max\{0, - v_i \deg(\Phi)\}$ are chosen
minimally, so that $\psi$ is a polynomial in $z_0,\ldots,z_{d-1}$. For each $d$, the symbol $K_d$ denotes the set of polynomials, which are products of extended
cyclotomic polynomials in $d$ variables and a monomial $\pm z_0^{c_0}\cdots z_{d-1}^{c_{d-1}}$.
\end{definition}
\end{sloppypar}

\begin{theorem}[Boyd \cite{Boyd80}, Smyth \cite{Smy81}] \label{th:Kron}
Let $P(z_0,\ldots,z_{d-1})$ be a polynomial with integer coefficients. Then $M(P)=1$ if and only if $P\in K_d$.
\end{theorem}

We will need also the well-known $L^2$ inequality for the Mahler measure:
\be \label{ineq:Mahler}
\mfr(P) \le \half \log\int_{\T^d} |P(\bz)|^2\,d\bt = \half\log\sum_{\bn \in \Z^d} |a_{\bn}|^2,
\ee
where $P(\bz) = \sum_{\bn \in \Z^d} a_\bn \bz^\bn$.


\section{Results}

\subsection{Lyapunov spectrum}
Denote by $|v|_j$ the number of letters $j$ in a word $v$. Then we can write
\be \label{eq1}
(E_\zeta \bz)_k= \exp\Bigl(-2\pi i \sum_{j=0}^{d-1} {|\zeta(k)|}_j \cdot \xi_j\Bigr) = \bz^{\zeta(k)}.
\ee

The following is essentially a variant of  \cite[Corollary and Lemma 3.1]{RS2}, which deal with the twisted cocycle for IET's.
See also \cite[Appendix]{RS2}, devoted to applications to substitution dynamics.
 The current formulation arose in discussions with Pedram Safaee.

\begin{proposition}[H. Rajabzade and P. Safaee] \label{prop-RS} \hspace{-10mm}.
\begin{enumerate}[label=\upshape(\roman*), leftmargin=*, widest=iii]
\item The cocycle $\Mc_\zeta$ has an invariant section: in fact, the vector-function 
$\Rk(\bz):= \left[ \begin{array}{c} 1-z_0 \\ \vdots \\ 1-z_{d-1}\end{array} \right]$ has the 
property $\Mc_\zeta(\bz) \Rk(\bz) = \Rk(E_\zeta \bz)$, hence \\[1.1ex] $\Mc_\zeta(\bz,n) \Rk(\bz) = \Rk(E_\zeta^n\bz)$ for $n\ge 1$.\label{it:1}
\smallskip
\item Under the assumption that $E_\zeta$ is ergodic, so that the Lyapunov exponents are constant a.e., one of the Lyapunov exponents of $\Mc_\zeta$ is zero.
\label{it:2}
\end{enumerate}
\end{proposition}

\begin{proof}
We provide a quick proof for the reader's convenience.

(i) Denote by $\Pref_k(v)$ the prefix of length $k$ of a word $v$, so that $\Pref_0(v)$ is the empty word and
$\Pref_{|v|}(v)$ is $v$ itself. Then we can write for $i\le d$:
\begin{eqnarray*}
\bigl(\Mc_\zeta(\bz) \Rk(\bz)\bigr)_i & = & \sum_{j=0}^{d-1} \ \sum_{k:\, \zeta(i)_k = j} \bz^{\Pref_{k-1}(\zeta(i))}(1-z_j) \\
& = & \sum_{k=1}^{|\zeta(i)|} \bz^{\Pref_{k-1}(\zeta(i))}(1-z_{\zeta(i)_k}) \\
& = & \sum_{k=1}^{|\zeta(i)|} \bigl(\bz^{\Pref_{k-1}(\zeta(i))} - \bz^{\Pref_{k}(\zeta(i))}\bigr) \\
& = & 1 - \bz^{\zeta(i)} \\
& = & \Rk(E_\zeta \bz)_i,
\end{eqnarray*}
in view of \eqref{eq1}, as claimed.

(ii) It is easy to see by a Borel-Cantelli argument that for $m_d$-a.e.\ $\bz\in \T^d$,
$$
\chi_{\zeta,\bz,\Rk(\bz)} = \lim_{n\to \infty} n^{-1} \log \|\Mc_\zeta(\bz,n) \Rk(\bz)\| = \lim_{n\to \infty} n^{-1} \log \| \Rk(E^n_\zeta\bz)\| = 0,
$$
since the function $\bz \mapsto \log\|\Rk(\bz)\|$ is continuous on the torus, except for a single logarithmic singularity at $\bz = (1,\ldots,1)$.
If $E_\zeta$ is ergodic, we obtain that there is a zero exponent in the direction of $\Rk(\bz)$ for $m_d$-a.e.\ $\bz$.
\end{proof}

\begin{remark}
1. The proposition extends to the $S$-adic case, with appropriate modifications.

2. If $\zeta$ is a constant length substitution of length $q\ge 2$, 
it is natural to restrict the cocycle to the diagonal, which is invariant under $E_\zeta$. Then the base dynamics is simply $\xi \mapsto q\xi$ (mod 1) on 
$\T^1 \cong \R/\Z$ and $\Rk(\xi)$ is an eigenvector of $\Mc_\zeta(\xi)$ for $\xi \ne 0$, parallel to $(1,\ldots,1)^{\sf T}$.
 In this case the existence of a zero Lyapunov exponent was shown in 
\cite[Theorem 4.2]{BCM20} for $d=2$ and in \cite[Prop.\ 3.3]{Marshall24a} for any $d\ge 2$. 
\end{remark}

As a corollary, we  obtain the following generalization of \cite[Theorem 1.1]{BCM20} from the case of constant-length substitutions on 2 symbols to 
the 2-symbol non-constant length case.

\begin{corollary} \label{cor:Lyap}
For any primitive substitution $\zeta$ on 2 symbols, with $\det(\Sf_\zeta)\ne 0$ and no eigenvalues $\pm 1$, the Lyapunov exponents are given by
$$
\chi_{_{\scriptstyle \zeta,\max}} = \mfr(p_\zeta),\ \ \ \chi_{_{\scriptstyle \zeta,\min}} = 0,
$$
where $p_\zeta(\bz) = \det\Mc_\zeta(\bz)$.
\end{corollary}

\begin{proof}
This is immediate from the (one-sided) Oseledets Theorem, which says, in particular, that the sum of the Lyapunov exponents of a matrix cocycle equals the 
integral of the logarithm of the modulus of the determinant, see e.g. \cite[Theorem 3.5.10]{BP}.
\end{proof}

\subsection{Application: singularity of the spectrum} 
Recall the following:

\begin{theorem}[{\cite[Theorems 2.4, 2.5]{BuSo22}}] \label{th:BuSo22}
\begin{enumerate}[label=\upshape(\roman*), leftmargin=*, widest=iii]
\item Let $\zeta$ be a primitive substitution on $d\ge 2$ symbols, such that $\Sf_\zeta$ is irreducible over
$\Q$. Let $\theta$ be the Perron-Frobenius eigenvalue of $\Sf_\zeta$. If
$$
\chi_{_{\scriptstyle \zeta,\max}}< \half \log \theta,
$$
then the substitution $\Z$-action has pure singular spectrum.
\label{it:1}
\item The same conclusion holds for reducible substitutions on 2 symbols, whose matrix $\Sf_\zeta$ has two integer eigenvalues $\theta=\theta_1 
>|\theta_2|>1$.\label{it:2}
\end{enumerate}
\end{theorem}

Note that irreducibility of the matrix $\Sf_\zeta$ over the rationals implies that the endomorphism $E_\zeta$ is ergodic (in any dimension).
Indeed, the criterion for ergodicity
of a toral endomorphism is absence of roots unity in the spectrum of the defining integer matrix, see e.g., \cite[Corollary 2.2]{EWergbook}.
Since the Perron-Frobenius eigenvalue $\Sf_\zeta$ is not a root of unity, non-ergodicity of $E_\zeta$ implies that the characteristic polynomial of $\Sf_\zeta$
is reducible over  $\Q$.

Theorem~\ref{th:BuSo22} was extended by Yaari \cite{Yaari} to the case when $\Sf_\zeta$ is reducible and also to the case of suspension flows under certain non-degeneracy conditions. In particular, for the statement on substitution $\Z$-actions, one needs to restrict the cocycle to the subtorus corresponding to the minimal
$\Sf_\zeta^{\sf T}$-invariant subspace over $\Q$, containing the vector $(1,\ldots,1)$.
For the closely related notion of diffraction spectrum, for the self-similar $\R$-action, analogous results were obtained earlier by Baake and
collaborators in \cite{BFGR19,BGM18,BGM19} by a different method.

Combining these results with Corollary~\ref{cor:Lyap} allows us to greatly expand the list of examples of non-constant length substitutions on 2 symbols, for which the spectrum is purely singular, avoiding lengthy numerical computations, as those used in \cite{BFGR19,BGM18}. 
In particular, it immediately follows that for the  family of substitutions $1\to 0\to 01^m$ studied in \cite{BGM18}, the Lyapunov spectrum is degenerate:
$\chi_{_{\scriptstyle \zeta,\max}} = \chi_{_{\scriptstyle \zeta,\min}} = 0$.

We think that it is very
unlikely that there exists a substitution system on 2 symbols with a Lebesgue spectral component, but this remains an open question. 

The next proposition contains a ``sample'' 
list of sufficient conditions for singularity, which is by no means exhaustive. 

\begin{proposition} \label{prop:sing}
Suppose that $\zeta$ is a substitution on 2 symbols, with the PF eigenvalue $\theta$, and $\Sf_\zeta^{\sf T}$ is irreducible over $\Q$.
Any of the following conditions implies that the substitution $\Z$-action has pure singular spectrum.
\begin{enumerate}[label=\upshape(\roman*), leftmargin=*, widest=iii]
\item \cite[Proposition 5.1]{RS2} Suppose that 
$\Sf_\zeta^{\sf T} = \left(\begin{array}{cc} A & B \\ C & D \end{array} \right)$, with $A, B, C, D >0$, such that
$\theta > 2 \min\{A+C, B+D\}$; \label{it:1} \smallskip
\item $\Sf_\zeta^{\sf T} = \left(\begin{array}{cc} A & B \\ C & 0 \end{array} \right)$,
with $A, B, C>0$, and $B < A+C$;\label{it:2} \smallskip
\item $\zeta(0) = 0^A 1^B,\ \zeta(1) = 1^C 0^D$, with $A, B, C, D >0$, such that $\theta> 6$.\label{it:3}
\end{enumerate}
\end{proposition}

\begin{remark} 
If $\zeta$ is a Pisot substitution, that is, the second eigenvalue of $\Sf_\zeta$ lies inside the unit circle, then the substitution $\Z$-action has pure discrete,
hence singular, spectrum, see \cite{BD,HS}. 
For matrices as in part (ii), the Pisot condition is equivalent to  $A \ge BC$.
\end{remark}

 The proof of Proposition~\ref{prop:sing}(i) appeared in \cite[Appendix]{RS2}. 
We repeat it here in our notation, since it will be useful later.

\begin{lemma}[{See the proof of \cite[Prop.\,5.1]{RS2}}] \label{lem:RS}
Suppose that $\zeta$ is a primitive substitution of $\Ak=\{0,1\}$.
Then
\begin{eqnarray*}
\mfr\bigl(\det\Mc_\zeta(z_0,z_1)\bigr) & = & \mfr\bigl(P_{\zeta(01),0}(z_0,z_1) - P_{\zeta{(10)},0}(z_0,z_1)\bigr) \\ & = & 
\mfr\bigl(P_{\zeta(01),1}(z_0,z_1) - P_{\zeta{(10)},1}(z_0,z_1)\bigr).
\end{eqnarray*}
\end{lemma}

\begin{proof}
Suppose that $\Sf_\zeta^{\sf T} = \left(\begin{array}{cc} A & B \\ C & D \end{array} \right)$, and 
let $\bbe_1=(1,0)^{\sf T}$. We have, in view of Proposition~\ref{prop-RS},
\begin{eqnarray*}
\det\Mc_\zeta(z_0,z_1) & = & 
                              \frac{\det\bigl(\Mc_\zeta(z_0,z_1)\bbe_1, \Mc_\zeta(z_0,z_1) \Rk(z_0,z_1)\bigr)}{\det(\bbe_1,\Rk(z_0,z_1))} \\[1.2ex]
                              & = & \frac{\det \left(\begin{array}{cc} P_{\zeta(0),0}(z_0,z_1) & 1 - z_0^A z_1^B \\ P_{\zeta(1),0}(z_0,z_1) & 1 - z_0^C z_1^D\end{array}
                              \right)}{1-z_1}\,.
\end{eqnarray*}
Since the Mahler measure of $1-z_1$ is zero, we obtain that
$\mfr\bigl(\det\Mc_\zeta(z_0,z_1)\bigr) = \mfr\bigl(Q_1(z_0,z_1))$,
where
\begin{eqnarray}
Q_1(z_0,z_1) &  = & P_{\zeta(0),0}(z_0,z_1) + z_0^Az_1^B P_{\zeta(1),0}(z_0,z_1) \nonumber \\ & & - P_{\zeta(1),0}(z_0,z_1) - z_0^C z_1^D P_{\zeta(0),0}(z_0,z_1) 
\nonumber\\
& = & P_{\zeta(01),0}(z_0,z_1) - P_{\zeta{(10)},0}(z_0,z_1). \label{def:Q1}
\end{eqnarray}
The second equality follows in the same way, using $\bbe_2$.
\end{proof}

\begin{proof}[Proof of Proposition~\ref{prop:sing}]
(i) We are going to use the inequality \eqref{ineq:Mahler} for the polynomial $Q_1$ in \eqref{def:Q1}. Observe that this polynomial has coefficients
$-1,0,1$, hence it suffices to estimate the number of monomials in $Q_1$, which is at most $2(A+C)$. Thus, 
$\mfr(p_\zeta) \le  \half \log(2(A+C))$  by \eqref{ineq:Mahler}, and
similarly, $\mfr(p_\zeta) \le  \half\log(2(B+D))$, by the 2nd equality in Lemma~\ref{lem:RS}. Now the claim follows from Theorem~\ref{th:BuSo22}.

(ii)  Observe that
$$
\mfr(p_\zeta(\bz)) = \mfr(P_{U,1}(\bz)),\ \ \mbox{with}\ U=\zeta(0).
$$
But $P_{U,1}(\bz)$ has exactly $B$ monomials with coefficients 1, hence $\mfr(p_\zeta(\bz))\le \half \log B$ by \eqref{ineq:Mahler}.
However, $\theta>B$ when  $A+C>B$, and an application of Theorem~\ref{th:BuSo22} finishes the proof, as above.

(iii)
Now $\zeta(0) = 0^A 1^B,\ \zeta(1) = 1^C 0^D$, with $A, B, C, D >0$.  Denote
$$
\Psi_m(t) := \sum_{j=0}^{m-1} t^j,\ \ \ m\ge 1.
$$
We have
$$
p_\zeta(\bz) = \det\Mc_\zeta(\bz) = \Psi_A(z_0)\Psi_C(z_1) - z_0^A z_1^C \Psi_B(z_1) \Psi_D(z_0).
$$
Then $p_\zeta(\bz)(1-z_0)(1-z_1)$ is a polynomial with 6 monomials, and assuming $\theta >6$, we can conclude as in (i) and (ii).
\end{proof}


\section{Positivity of the top exponent}
A related question that has been raised is to determine when the top Lyapunov exponent $\chi_{_{\scriptstyle \zeta,\max}}$ is positive
(respectively zero). Recently it was shown \cite{RS1} that for the twisted cocycle associated with interval exchange transformations, corresponding to a translation surface of genus $g\ge 2$, the top Lyapunov exponent is positive with respect to the natural invariant measure.

In the 2-symbol substitution case,  positivity of the top exponent  is equivalent to $\mfr(p_\zeta)>0$, in view of Corollary~\ref{cor:Lyap}.
For binary (i.e., 2-symbol) substitutions of constant length $q$, when the cocycle is over the times-$q$ map, a variety of results on (non)-degeneracy of the
Lyapunov spectrum were obtained in \cite[Sections 3.2-3.4]{Man-thesis}. Here we focus on non-constant length substitutions.
Theorem~\ref{th:Kron} provides an answer to the question when the Lyapunov spectrum is degenerate; it is not always easy to check in concrete cases.
The following are partial results in this direction.

\begin{lemma} \label{lem:con1}
Suppose that $P(z_0,z_1)$ is a polynomial which has the term of lowest degree equal to 1. Suppose,
moreover, that $\mfr(P) = 0$, so that it is a product of extended cyclotomic polynomials from Definition~\ref{def:cyc}. Then
every polynomial in this factorization has 1 as its lowest term and can be written as $\Phi(z_0^{v_0} z_1^{v_1})$, where
$v_0, v_1\ge 0$, $(v_0,v_1)\ne (0,0)$, hence $b_0 = b_1=0$
in \eqref{eq:cyc}. 
\end{lemma}

\begin{proof}
It follows from the assumption
that every extended cyclotomic polynomial, appearing in the product, must have the term of lowest degree equal to $\pm 1$.
Let $\psi(\bz) = z_0^{b_0} z_1^{b_1} \Phi(z_0^{v_0} z_1^{v_1})$ be one of them, with
$\Phi$ a cyclotomic polynomial, $v_i$ integers, and $b_i= \max\{0, - v_i \deg(\Phi)\}$.
Let $m=\deg(\Phi)$. If $v_0<0$ and $v_1<0$, then 
$$
\psi(\bz) = z_0^{-mv_0} z_1^{-mv_1} \Phi(z_0^{v_0} z_1^{v_1}) = \pm \Phi(z_0^{-v_0} z_1^{-v_1}),
$$
since every cyclotomic polynomial is palindromic: $\Phi(z) = \Phi(z^{-1})z^m$, except for $P_1(z) = z-1$, for which the equality holds up to the minus sign.

Now suppose that $v_0 \ge 0, v_1<0$; then  $\psi(\bz) = z_1^{m|v_1|} \Phi(z_0^{v_0} z_1^{-|v_1|})$. In order to have a term $\pm 1$ in $\Phi(\bz)$, 
we must have $v_0=0$. But then, as above, we can write $\psi(\bz) = z_1^{m|v_1|} \Phi(z_1^{-|v_1|}) = \pm \Phi(z_1^{|v_1|})$. The remaining case
follows by symmetry.
\end{proof}

\subsection{A special case}
Suppose first that $\zeta$ is a primitive aperiodic substitution on $\Ak = \{0,1\}$, with
$\Sf_\zeta^{\sf T} = \left(\begin{array}{cc} A & B \\ C & 0 \end{array} \right)$, that is,
$\zeta(0)  = U\in \Ak^*,\ \zeta(1) = 0^C,\  C\ge 1$. 
Recall that
$
\Psi_m(t) := \sum_{j=0}^{m-1} t^j,
$
which, of course, has zero Mahler measure.
We have
\be \label{ek0}
p_\zeta(\bz) = \det\Mc_\zeta(\bz) = -\Psi_C(z_0)P_{U,1}(\bz),
\ee
where $P_{U,1}(\bz)$ was defined in \eqref{poly-return}. Thus it suffices to decide when  $\mfr(P_{U,1}(\bz))=0$. If $U = 0^{n_1}\wt U0^{n_2}$, then
$P_{U,1}(\bz) = z_0^{n_1} P_{\wt U,1}(\bz)$, hence we can assume without loss of generality that $U$ starts and ends with a 1. Then
$$
P_{U,1}(\bz) = 1+\sum_{j=1}^{N-1} z_0^{n_j} z_1^j,
$$
where $N$ is the number of 1's in $U$  and $n_j$ is a non-decreasing sequence. In particular, 
\be \label{factor1}
P_{U,1}(1,z_1) = \Psi_N(z_1).
\ee 

\begin{proposition} \label{prop-zero0}
Suppose that $U \in \{0,1\}^+$ is such that $\mfr(P_{U,1})=0$, and $U$ starts and ends with a 1. 
By Theorem~\ref{th:Kron}, $P_{U,1}$ is a product of extended cyclotomic polynomials. Then the following holds:
\begin{enumerate}[label=\upshape(\roman*), leftmargin=*, widest=iii]
\item If $U$ starts with $1^{k_1}0$, with $k_1\ge 2$, then $\Psi_{k_1}(z_1)$ is one of these factors, and all other factors are cyclotomic polynomials in monomials
with positive powers of both  $z_1$ and $z_0$. 
Moreover, $U$ can be written as a word $V$ in the alphabet $\{0,1\}$, with every 1 being replaced by 
$1^{k_1}$, and
\be \label{factor2}
P_{U,1}(z_0,z_1)= \Psi_{k_1}(z_1) P_{V,1}(z_0,z_1^{k_1}).
\ee
\label{it:1}
\item $U$ is a palindrome.\label{it:2}
\end{enumerate}
\end{proposition}

\begin{remark}
It seems plausible that if $\mfr(P_{U,1})=0$ and $U$ starts and ends with a 1, then $U$ may be represented inductively as follows:

{\em There exist $n\ge 1$ and two integer sequences $k_1,\ldots, k_n; \ \ell_1,\ldots, \ell_{n-1}$, such that $k_1\ge 1$, $k_j\ge 2$ for $j\ge 2$;
$\ell_1\ge 1,\ \ell_j\ge 0$ for $j\ge 2$, with $\ell_j \ne \ell_{j-1}$, and
$$
U_1 = 1^{k_1},\ \ U_2 = (U_1 0^{\ell_1})^{k_2-1} U_1,\cdots, U =  U_n = (U_{n-1}0^{\ell_{n-1}})^{k_n-1} U_{n-1}.
$$
Then $P_{U,1}(\bz) = P_{U_n,1}(\bz)$ is a product of extended cyclotomic polynomials, which can be computed inductively, as follows:
\be \label{poly1}
P_{U_1,1}(\bz) = \Psi_{k_1}(z_1);\ \ \ P_{U_j,1}(\bz) = P_{U_{j-1},1}(\bz)\cdot \Psi_{k_j}(z_1^{k_1\cdots k_{j-1}} z_0^{N_{j-1}}),\ j\ge 2,
\ee
with
\be \label{def-Nn}
N_1 = \ell_1,\ \ \ N_j = N_{j-1}k_{j} + \ell_{j} - \ell_{j-1},\ \ j\ge 2.
\ee
}
Thus, this is a sufficient condition to have $\mfr(P_{U,1})=0$. We do not know if it is necessary.
\end{remark}

\begin{example} 
(a) Let $n=3$; $k_1 = k_2 = 2,\ k_3 = 3$; $\ell_1 = 1,\ \ell_2 = 0$. Then
$$
U=U_3 = (11011)(11011)(11011),\ \ \ P_{U,1}(\bz) = (1 + z_1)(1 + z_1^2 z_0)(1 + z_1^4 z_0 + z_1^8 z_0^2).
$$
The parentheses in $U$ are just in order to see the structure of the word better.

\smallskip

(b) Let $n=4$;\ $k_1 = 1,\ k_2 = k_3 = k_4 = 2$;\ $\ell_1 =2,\ \ell_2 = 1,\ \ell_3 = 3$. Then
$$
U = U_4 = (1001)0(1001)000(1001)0(1001),
$$
$$
P_{U,1}(\bz) = (1+z_1 z_0^2)(1 +z_1^2 z_0^3)(1+z_1^4 z_0^8) = \bigl(1+z_1z_0^2 + (z_1z_0^2)^4 + (z_1z_0^2)^{5}
\bigr)(1 +z_1^2 z_0^3).
$$
\end{example}

\begin{proof}[Proof of Proposition~\ref{prop-zero0}] Let $\Phi_m$ be the $m$-th  cyclotomic polynomial. Then $\Phi_1(t) = t-1$, and for all 
$m\ge 2$ the polynomial $\Phi_m$ is reciprocal: $\Phi_m(t) = t^{\deg \Phi_m} \Phi_m(t^{-1})$, see, e.g., \cite{Sanna}.

Note that $\Phi_1(z_0^{v_0} z_1^{v_1}) =z_0^{v_0} z_1^{v_1}-1$ cannot appear in the factorization of $P_{U,1}(z_0,z_1)$, since fixing $z_0=1$ we obtain
a factorization of $1+z_1+\cdots + z_1^{k-1}$, which does not have $z_1=1$ as a root.

\smallskip

The term of lowest degree in $P_{U,1}(\bz)$ is equal to 1.  It follows that every extended cyclotomic polynomial, appearing in the factorization, must have the term of lowest degree equal to $1$.
Let $\psi(\bz) = z_0^{b_0} z_1^{b_1} \Phi(z_0^{v_0} z_1^{v_1})$ be one of them, with
$\Phi$ a cyclotomic polynomial, $v_i$ integers, and $b_i= \max\{0, - v_i \deg(\Phi)\}$. By Lemma~\ref{lem:con1} we can assume that 
$b_0 = b_1 =0$.

\smallskip

(i) Suppose that $U$ starts with $1^{k_1}0$, with $k_1\ge 2$. Then $P_{U,1}(\bz)$ starts with $1 + z_1 +\cdots + z_1^{k_1-1} = \Psi_{k_1}(z_1)$, 
followed by $z_0^{\ell_1} z_1^{k_1}$ for some $\ell_1\ge 1$. Note that 
$\Psi_{k_1}(z_1)$ is not irreducible if $k_1$ is not prime, but it can only appear as a product of irreducible cyclotomic polynomials in $z_1$ (and not in monomials with $z_0$ present), hence it has to be one of the factors. It follows from \eqref{factor1} that $N = |U|_1=
k_1 k_2$ for some $k_2\ge 2$, and then
$$
P_{U,1}(1,z_1) = \Psi_{k_1k_2}(z_1) = \Psi_{k_1}(z_1)\Psi_{k_2}(z_1^{k_1}),
$$
which implies \eqref{factor2}.

\smallskip

(ii) It is easy to see that $U$ is palindromic (under our assumption that $U$ starts and ends with a 1) if and only if 
$$
P_{U,1}(z_0,z_1) = z_0^{|U|_0} z_1^{|U|_1-1} P_{U,1}(z_0^{-1},z_1^{-1}).
$$
This holds under our assumptions,
since all the factors in the product are reciprocal polynomials in monomials of the form $z_0^{v_0}z_1^{v_1}$.
\end{proof}


\subsection{General case} Now suppose that $\Sf_\zeta^{\sf T} = \left(\begin{array}{cc} A & B \\ C & D \end{array} \right)$, with $A, B, C, D >0$.
Here we just collect a few ad hoc observations. The following reduction is sometimes useful:

\begin{itemize}
\item if $\zeta(0) = WU,\ \zeta(1)=W$, then
\be \label{eq:reduce}
p_\zeta(\bz) = \bz^{W} p_{\zeta'}(\bz),\ \ \mbox{where}\ \ \zeta'(0) = U,\ \zeta'(1) = W;
\ee

\item  if  $\zeta(0) = UW,\ \zeta(1) = W$, then
$$
p_\zeta(\bz) =  p_{\zeta'}(\bz),\ \ \mbox{where}\ \ \zeta'(0) = U,\ \zeta'(1) = W.
$$
\end{itemize}
\vspace{-2mm}
There is, of course, a symmetric case, obtained by exchanging 0 and 1. The proof is a straightforward verification, left to the reader.

\smallskip

Next we are going to use Lemma~\ref{lem:RS}, which says that
$$
\mfr(p_\zeta(\bz))= \mfr(Q_{\zeta,0}(\bz)) = \mfr(Q_{\zeta,1}(\bz)),
$$
where
$$
Q_{\zeta,0}(\bz) = P_{\zeta(01),0}(\bz) - P_{\zeta{(10)},0}(\bz)\ \ \ \mbox{and}\ \ \ 
Q_{\zeta,1}(\bz) = P_{\zeta(10),1}(\bz) - P_{\zeta{(01)},1}(\bz).
$$

Now suppose that $\zeta(0)$ and $\zeta(1)$ have no common initial word (of course, everything can be repeated with a common terminal word).
Without loss of generality, passing to $\zeta^2$, we can assume that $\zeta(0)$ starts with $0$ and $\zeta(1)$ starts with $1$. In this case
both polynomials $Q_{\zeta,0}(\bz)$ and $Q_{\zeta,1}(\bz)$ have the constant term equal to 1. If, say, $\mfr(Q_{\zeta,0}(\bz))=0$, then
Lemma~\ref{lem:con1} applies to this polynomial, and we obtain, in particular, that $Q_{\zeta,0}(z,z)$ must be a product of cyclotomic polynomials
of one variable. 

\smallskip

One might think that pure discrete spectrum
of the substitution measure-preserving system is related to the positivity of the maximal Lyapunov exponent, but this is not the case. As already mentioned
above, any Pisot substitution system on two symbols has pure discrete spectrum. Note that if the substitution $\zeta$ is unimodular, that is, 
$\det(\Sf_\zeta) = \pm 1$, then it is automatically of Pisot type in the two-symbol case, since the Perron-Frobenius eigenvalue is greater than one.
Using the observations above, it is easy to construct examples of unimodular Pisot substitutions for which changing the order of letters in the words 
$\zeta(0), \zeta(1)$ changes the
maximal Lyapunov exponent from zero to a positive number. For instance, consider the following two substitutions with the same unimodular Pisot substitution matrix:
$$
\zeta_1(0) = 01001,\ \zeta_1(1) = 010;\ \ \ \ \ \zeta_2(0) = 00011,\ \zeta_2(1) = 100.
$$
Observe that $\zeta_1 = \zeta_{\rm F}^3$, where $\zeta_{\rm F}(0) = 01,\ \zeta_{\rm F}(1) = 0$ is the Fibonacci substitution. This immediately 
yields that $\mfr(p_{\zeta_1}(\bz)) = 0$, hence $\chi_{\zeta_1,\max}=0$.
Alternatively, using \eqref{eq:reduce} one can see that $p_{\zeta_1}(\bz) = -z_0^4 z_1^2$, which yields the
same conclusion. On the other hand,
$$
Q_{\zeta_2,1} (\bz) = 1 - z_0^3 - z_0^3z_1 - z_0^3z_1^2 + z_0^5 z_1 + z_0^5z_1^2,
$$
and $P(z) = Q_{\zeta_2,1} (z,z) = 1-z^3-z^4-z^5+z^6+z^7$ is such that $P(z) \ne \pm z^{\deg P} \cdot P(1/z)$.
Hence $\mfr(Q_{\zeta_2,1} (\bz))>0$, and we get that $\chi_{\zeta_2,\max}>0$.


\appendix
\section{Dimension of the spectral measure for irrational rotations}

It was recently shown in \cite[Corollary, p.3]{RS2} that for IET's, corresponding to a permutation of rotation type,
the Lyapunov spectrum of the twisted cocycle with respect to the natural invariant measure is totally degenerate, i.e., all its exponents are equal to zero. On the other hand, in \cite{BuSo20} it was proven that under appropriate 
assumptions, in particular, for primitive substitutions and ``well-behaved'' $S$-adic systems, zero Lyapunov exponents imply that 
the local dimension of spectral measures of cylindrical functions is greater or equal to $2$ (see \cite[Theorem 4.3]{BuSo20}).
Below we demonstrate that this is actually an equality for Lebesgue almost all irrational rotations, by a direct computation.
This should not be interpreted to mean that studying local dimensions is not interesting; in fact, they are very  useful, e.g., in the proofs of singular
spectrum. Moreover, the case of rotations is very special.

\medskip

Consider the irrational rotation on the circle $(\T, R_\theta, m)$. We view $\T$ additively, as $\R/\Z$, and $R_\theta(x) = x+\theta$ (mod 1). It has pure discrete spectrum, with the eigenfunctions 
$\bbe_n = e^{2\pi i nx},\ n\in \Z$, and the corresponding eigenvalues $\lam_n = e^{2\pi i n \theta}$. For $f\in L^2(\T,m)$ we have the usual Fourier 
expansion
$$
f = \sum_{n\in \Z} c_n \bbe_n, \ \ \mbox{with}\ c_n = \what f(n) = \int_0^1 f(x) e^{-2\pi inx}\,dx,\ \ \mbox{and}\ \ \|f\|^2 = \sum_{n\in \Z} |c_n|^2.
$$
On the other hand, the spectral measure $\sig_f$ is determined by
$$
\what\sig_f(-k) = \int_0^1 e^{2\pi i k\om}\,d\sig_f(\om) = \langle U^k_T f, f \rangle = \sum_{n\in \Z} |c_n|^2 e^{2\pi ink\theta}.
\ \ k\in \Z,
$$
Writing $\{x\}$ for the fractional part of $x\in \R$ we obtain
\be \label{eq:meas}
\sig_f = \sum_{n\in \Z} |c_n|^2 \delta_{\{n\theta\}}.
\ee
Let $f = \One_{[0,1-\theta)}$, which corresponds to a simple cylinder function in the symbolic representation of the rotation as a Sturmian $S$-adic system,
see \cite{BeDel,Thus20}.
Then we have
\be \label{eq:coeff}
c_n = \int_0^{1-\theta} e^{-2\pi i nx}\,dx = \frac{1-e^{2\pi i n\theta}}{2\pi i n}.
\ee

\begin{theorem}
For Lebesgue-a.e.\ $\theta$, for Lebesgue-a.e.\ $x\in \T$, holds
$$
d(\sig_f,x):= \lim_{r\to 0} \frac{\log \sig_f(B(x,r))}{\log r} = 2.
$$
\end{theorem}

\begin{proof}
We consider irrationals of type 1, i.e., $\theta\in (0,1)$ such that
$$
\eta = \sup\{t>0: j^t \liminf_{j\to\infty} \|j\theta\| = 0\}=1.
$$
Here and below we denote by $\|t\|$ the distance from $t$ to the nearest integer. 
By the classical results of Khintchine (see e.g. \cite[Chapter VII]{Cassels}), the set of irrationals of type 1 has full Lebesgue measure.
Fix such a $\theta\in (0,1)$. It is enough to prove the claim for 
Lebesgue-a.e.\ $x\in (\eps,1-\eps)$, for an arbitrary small $\eps>0$. Fix $\eps>0$, and let $r\in (0,\eps/2)$. It follows from \eqref{eq:meas} 
and \eqref{eq:coeff} that
\be \label{eq:meas2}
x\in (\eps, 1-\eps) \implies \sig_f(B(x,r)) \asymp \sum_{n\in \Z:\ \{ n\theta \} \in B(x,r)} n^{-2}.
\ee
Here and below writing $A\asymp B$ means that $C^{-1}B \le A \le CB$ for some constant $C>1$, which may depend on $\theta$ and $\eps$.

We start with the upper bound for the local dimension, which corresponds to the lower bound for $\sig_f(B(x,r))$. Let $\Qk_n$ be the partition of $[0,1)$
by the orbit $\bigl\{\{ k\theta\}\bigl\},\ 0\le k \le n$. Denote by $Q_n(x)$ the element of $\Qk_n$ to which $x$ belongs.
It is a special case of \cite[Theorem 1.1]{KimPark08} (quite likely follows from some earlier results as well) that $\eta=1$ implies
$$
\lim_{n\to \infty} \frac{-\log m(Q_n(x))}{\log n} = 1\ \ \mbox{for $m$-a.e.}\ x\in [0,1),
$$
where $m$ is the Lebesgue measure.
Let $\tau>0$ be arbitrary. We obtain that for Lebesgue-a.e.\ $x\in (\eps,1-\eps)$, for all $n\ge n_0(x,\tau)$ there exists $k\in [0,n]$ such that
$$
\bigl|\{ k\theta \} - x\bigr| < n^{-(1-\tau)},
$$
hence, choosing $n\in \N$ so that $n^{-(1-\tau)}< r \le (n-1)^{-(1-\tau)}$, we obtain
$$
\sig_f(B(x,r))\ge n^{-2} \ge \const\cdot r^{\frac{2}{1-\tau}}.
$$
Since $\tau>0$ was arbitrary, it follows that for a.e.\ $x$,
$$
\lim\sup_{r\to 0} \frac{\log \sig_f(B(x,r))}{\log r} \le 2.
$$

Now we turn to the lower bound for the dimension, that is, to the upper bound for $\sig_f(B(x,r))$. For simplicity, we only estimate the contribution of the forward rotation orbit $\{ n\theta\}$, $n\ge 0$, in \eqref{eq:meas2}. The contribution of the negative orbit is estimated in exactly the same way.

Let $\tau>0$. Since $\theta$ is of type 1, there exists
$k_0 = k_0(\theta,\tau)$, such that
$$
\|k\theta\|\ge \frac{c}{k^{1+\tau}}\ \ \ \mbox{for all}\ \ k\ge k_0.
$$
Let $\delta>0$ be such that the orbit $\{\{ k\theta \},\ 0\le k \le k_0\}$ is $\delta$-separated, and suppose that $r\in (0, \delta/2)$.
It follows that if $n_2 > n_1 \ge k_0$ are such that
$\{ n_1\theta \},\ \{ n_2\theta \}\in B(x,r)$, then
$$
\|(n_2-n_1)\theta\| < 2r < \delta \implies n_2-n_1\ge k_0,
$$
and then
$$
2r > \|(n_2-n_1)\theta\| \ge \frac{c}{(n_2-n_1)^{1+\tau}}.
$$
We obtain that
\be \label{eka1}
n_2 - n_1 \ge \left(\frac{c}{2r}\right)^{1/1+\tau} = \wtil c r^{-1/1+\tau}.
\ee
On  the other hand, for a.e.\ $x\neq \{ n\theta \}$, $n\in \N$, by Borel-Cantelli, there exists $N = N_x$ such that 
$$
x\notin \bigcup_{n=N}^\infty B\Bigl(\{ n\theta \}, \frac{1}{n^{1+\tau}}\Bigr).
$$
We can assume that
$$
0 < r < \min_{n\le N} |x - \{ n\theta \}|.
$$
Let $n_0\in \N$ be minimal such that $\{ n_0\theta \} \in B(x,r)$. Then
\be \label{eka2}
r \ge n_0^{-(1+\tau)} \implies n_0 \ge r^{-\frac{1}{1+\tau}}.
\ee
Combining \eqref{eka1} and \eqref{eka2} we obtain that 
$$
 \sum_{n\in \Z:\ \{ n\theta \} \in B(x,r)} n^{-2} \le \sum_{\ell=0}^\infty \bigl(r^{-1/1+\tau}  +\wtil c \ell r^{-1/1+\tau}\bigr)^{-2}
= r^{2/1+\tau} \sum_{\ell=0}^\infty (1 + \wtil c\ell)^{-2} < C r^{2/1+\tau}.
$$
The contribution of $n<0$ is estimated similarly, and we obtain that for a.e.\ $x$,
$$
\lim\inf_{r\to 0} \frac{\log \sig_f(B(x,r))}{\log r} \ge 2,
$$
concluding the proof.
\end{proof}


\subsection*{Acknowledgements} I am grateful to Pedram Safaee for many interesting discussions, which inspired this paper, as well as for showing me some
of the proofs prior to their appearance in \cite{RS2}. Thanks to Chris Smyth for a helpful comment on cyclotomic polynomials and to the anonymous referee for many suggestions which improved the article. This research was supported by the Israel Science Foundation grant \#1647/23.



\bigskip

\begin{thebibliography}{99}

\bibitem{AFS24} Avila, Artur; Forni, Giovanni; Safaee, Pedram,
{\em Quantitative weak mixing for interval exchange transformations},
Geom. Funct. Anal. 33 (2023), 1--56.

\bibitem{BCM20}
Baake, Michael; Coons, Michael; Ma\~nibo, Neil, {\em
Binary constant-length substitutions and Mahler measures of Borwein polynomials}, in: 
Springer Proc. Math. Stat., 313
Springer, Cham, 2020, 303--322.

\bibitem{BFGR19} Baake, Michael; Frank, Natalie Priebe; Grimm, Uwe; Robinson, E. Arthur, Jr., {\em
Geometric properties of a binary non-Pisot inflation and absence of absolutely continuous diffraction}, 
Studia Math. 247 (2019), 109--154.

\bibitem{BGM18} Baake, Michael; Grimm, Uwe; Ma\~nibo, Neil, {\em Spectral analysis of a family of binary inflation rules},
Lett.\ Math.\ Phys. 108 (2018), 1783--1805.

\bibitem{BGM19} Baake, Michael; G\"ahler, Franz; Ma\~nibo, Neil, {\em 
 Renormalisation of pair correlation measures for primitive inflation rules and absence of absolutely continuous diffraction},
 Comm. Math. Phys. 370 (2019), no.\ 2, 591--635.
 
  \bibitem{BD} Barge, Marcy; Diamond, Beverly, {\em Coincidence for substitutions of Pisot type},
  Bull. Soc. Math. France 130\,(4) (2002), 619--626.
 
 \bibitem{BP} Barreira, Luis; Pesin, Yakov,
 {\em Nonuniform hyperbolicity}, Cambridge University Press, Cambridge, 2007.
 
 \bibitem{BeDel} Berth\'e, Valerie; Delecroix, Vincent, {\em
Beyond substitutive dynamical systems: $S$-adic expansions}, in: 
Numeration and substitution 2012, 81--123, RIMS K\^{o}ky\^{u}roku Bessatsu, B46, Res. Inst. Math. Sci. (RIMS), Kyoto, 2012.
  
\bibitem{BuSo14} Bufetov,  Alexander I.; Solomyak, Boris, {\em
On the modulus of continuity for spectral measures in substitution dynamics}, Adv.\ Math. 60 (2014), 84--129.
 
 \bibitem{BuSo18} Bufetov,  Alexander I.; Solomyak, Boris,
{\em The H\"older property for the spectrum of translation flows in genus two}, 
Israel J.\ Math. 223 (2018), 205--259.

\bibitem{BuSo20}  Bufetov,  Alexander I.; Solomyak, Boris. {\em 
A spectral cocycle for substitution systems and translation flows}, J. Anal. Math. 141 (2020), 165--205.

\bibitem{BuSo22} Bufetov,  Alexander I.; Solomyak, Boris. {\em 
On singular substitution $\Z$-actions.}, Math. Z. 301 (2022), no.\ 2, 1315--1331.

\bibitem{BuSo23} Bufetov,  Alexander I.; Solomyak, Boris, {\em 
 Self-similarity and spectral theory: on the spectrum of substitutions},
St.\ Petersburg Math.\ J. 34 (2023), 313--346.

\bibitem{Boyd80}
Boyd, David W, {\em Kronecker's theorem and Lehmer's problem for polynomials in several variables},
J.\ Number Theory 13 (1981), 116--121.

\bibitem{Cassels}
Cassels, J. W. S.,
{\em An introduction to Diophantine approximation},
Cambridge University Press, New York, 1957.

\bibitem{EW} Everest, Graham; Ward, Thomas, {\em Heights of Polynomials and Entropy in Algebraic Dynamics},
 Springer, London (1999).

\bibitem{EWergbook} Einsiedler, Manfred; Ward, Thomas,
{\em Ergodic theory with a view towards number theory},
Graduate Texts in Mathematics, 259. Springer-Verlag London, Ltd., London, 2011.

\bibitem{Siegel} Fogg,  N. Pytheas, {\em Substitutions in dynamics, arithmetics and combinatorics}, V. Berth\'e, S. Ferenczi, C. Mauduit and A. Siegel (eds.),
  Lecture Notes in Math. 1794, Springer, Berlin, 2002.

\bibitem{Forni24} Forni, Giovanni, {\em 
Effective Unique Ergodicity and Weak Mixing of Translation Flows},  In: Bonanno, C., Sorrentino, A., Ulcigrai, C. (eds) 
 Modern Aspects of Dynamical Systems, Lecture Notes in Math. 2347, Springer, Cham.

\bibitem{Forni22} Forni, Giovanni, {\em
Twisted translation flows and effective weak mixing}, J.\ Eur.\ Math.\ Soc.\ (JEMS) 24 (2022), 4225--4276.

\bibitem{FK} Furstenberg, Harry; Kesten, Harry, {\em Products of Random Matrices}, 
 Ann.\ Math.\ Statist. 31 (1960), 457--469.
 
  \bibitem{HS} Hollander, Michael; Solomyak, Boris, {\em 
Two-symbol Pisot substitutions have pure purely discrete spectrum}, 
Ergodic Theory Dynam.\ Systems 23 (2003), no.\ 2, 533--540.

 
 \bibitem{KimPark08}
 Kim, Dong Han; Park, Kyewon Koh, {\em 
The first return time properties of an irrational rotation}, 
Proc. Amer. Math. Soc. 136 (2008), no. 11, 3941--3951.
  
\bibitem{Kingman} Kingman, John F. C., {\em 
The ergodic theory of subadditive stochastic processes}, 
J.\ Roy.\ Statist.\ Soc.\ Ser.\ B 30 (1968), 499--510.

\bibitem{Mahler} Mahler, Kurt, {\em On some inequalities for polynomials in several variables},
J.\ London Math.\ Soc. 37 (1962), 341--344.

\bibitem{Man-thesis} Ma\~nibo, Neil, {\em Lyapunov Exponents in the  Spectral Theory of Primitive Inflation Systems}, Ph.D. Thesis, University of Bielefeld, Germany, 2019.

\bibitem{Marshall24a}
Marshall-Maldonado, Juan,  {\em 
Lyapunov exponents of the spectral cocycle for topological factors of bijective substitutions on two letters}, 
Discrete Contin. Dyn. Syst. 44 (2024), no. 7, 2068--2092.

\bibitem{Marshall24b}
Marshall-Maldonado, Juan, {\em 
Modulus of continuity for spectral measures of suspension flows over Salem type substitutions},
Israel J. Math. 263 (2024), no. 2, 965--1000.

\bibitem{Queff} Queffelec, Martine, {\em Substitution Dynamical Systems - Spectral Analysis}, Second edition. Lecture Notes in Math., 1294, Springer, Berlin, 2010.

\bibitem{RS1} Rajabzadeh, Hesam; Safaee, Pedram,
{\em Nondegeneracy of the spectrum of the twisted cocycle for interval exchange transformations}, 
arXiv:2309.05175 (2023).

\bibitem{RS2}  Rajabzadeh, Hesam; Safaee, Pedram, {\em Twisted cocycle for interval exchange transformations: invariant structures and Lyapunov 
spectrum},
arXiv:2501.16824.


\bibitem{Sanna} 
Sanna, Carlo,
{\em A survey on coefficients of cyclotomic polynomials}, Expo.\ Math.  40 (2022), no.\ 3, 469--49.

\bibitem{Smy81} Smyth, C. J., {\em 
A Kronecker-type theorem for complex polynomials in several variables}, 
Canadian Math. Bulletin  24 (1981), 137--149.

\bibitem{Sol24} Solomyak, Boris. {\em A note on spectral properties of random $S$-adic systems (with an Appendix by Pascal Hubert and Carlos Matheus)},
Pure Appl. Funct. Anal. 10 (2025), no. 2, 445?467.

\bibitem{Thus20} Thuswaldner, J\"org M, {\em 
$S$-adic sequences: a bridge between dynamics, arithmetic, and geometry}, in: 
Lecture Notes in Math. 2273, 
Springer, Cham, 2020, 97--191.

\bibitem{Yaari} Yaari, Rotem, {\em 
 Uniformly distributed orbits in  $\T^d$ and singular substitution dynamical systems}, 
Monatsh. Math. 201 (2023), 289--306.

\end{thebibliography}
\end{document}